\documentclass[12pt]{article}
\usepackage{amssymb}
\usepackage{eucal}
\usepackage{amsmath}
\usepackage{amsthm}
\usepackage{amssymb}
\usepackage{eucal}
\usepackage{amscd}
\usepackage{graphics}
\usepackage{graphicx}
\usepackage{amsfonts}
\usepackage{latexsym}
\usepackage[all]{xy}
\usepackage{enumerate}
\usepackage{framed}
\usepackage{stmaryrd}

\newcommand{\R}{{\mathbb  R}}

\numberwithin{equation}{section}
\newtheorem{thm}{\bf Theorem}[section]
\newtheorem{lem}[thm]{\bf Lemma}
\newtheorem{prop}[thm]{\bf Proposition}

\theoremstyle{remark}
\newtheorem{rem}{\bf Remark}[section]

\begin{document}

\title{Existence and Stability of 3-Cycles in Quadratic Maps}
\author{Dan Com\u anescu\\
{\small Department of Mathematics, West University of Timi\c soara}\\
{\small Bd. V. P\^ arvan, No 4, 300223 Timi\c soara, Rom\^ ania}\\
{\small E-mail address: dan.comanescu@e-uvt.ro}}
\date{}

\maketitle

\begin{abstract}
For a discrete dynamical system on $\R$ generated by a quadratic function, we show, using elementary computations, that the existence, number, and stability of 3-cycles are determined by a single parameter depending on the coefficients of the function.
\end{abstract}


\section{Introduction}

A large number of discrete dynamical systems, both theoretically significant and practically relevant, are defined on the real axis. Despite their apparent simplicity, discrete dynamical systems generated by quadratic functions have been demonstrated to produce highly complex dynamical behaviour. The most extensively studied case is that given by the logistic map. 
The pedagogical value of logistic map has been emphasized in several studies (see, for example, \cite{may} and \cite{saha}), and this observation applies equally to quadratic map. As noted in \cite{saha}, these dynamical systems are considered “accessible” and “exemplary,” illustrating “living mathematics” with clear relevance to scientific applications.

Interpreting chaos as sensitive dependence on initial conditions, Li and Yorke, in \cite{Li}, characterize the chaotic behaviour of a discrete dynamical system on the real line, generated by a continuous function with certain properties, by the existence of periodic points of every order within a finite interval. One consequence of the main result in \cite{Li} is that the existence of a 3-cycle implies the existence of an $N$-cycle for any natural number $N$. A more general version of this result had been established a few years earlier by Sharkovskii, see \cite{Sharkovskii}.

In \cite{katok} the {\it quadratic family} is defined as the collection of functions $g_{\lambda}:\R\to \R$, $g_{\lambda}(x)=\lambda x(1-x)$, where $\lambda$ is a real parameter. In this work, we consider a more general quadratic family consisting of all functions $g:\R\to \R$ of the form $g(x)=ax^2+bx+c$, where $a,b,c\in \R$ and $a\neq 0$.  

In Section 2 we show, using elementary computations, that for a discrete dynamical system on $\R$ generated by a quadratic function, the existence, number, and stability of 3-cycles are determined by a parameter - called the {\bf perturbed discriminant} - which depends on the coefficients of the function. Moreover, the points of any 3-cycle are the roots of a cubic polynomial whose coefficients are functions of the perturbed discriminant.  

In Section 3, we apply the theoretical results developed in Section 2 to two classical families of quadratic maps generated by the following quadratic functions: 
\begin{enumerate}[(I)]
\item $g_c(x)=x^2+c$, where $c\in \R$;  
\item $g_{\lambda}(x)=\lambda x(1-x)$, where $\lambda\in \R$ (logistic maps).   
\end{enumerate}
The dynamics of these maps have been extensively studied in the literature. 

In case (I), we recover the results on the existence and stability of 3-cycles, which can be obtained through elementary computations, as highlighted in the recent work \cite{casu}. 

For the logistic map, our approach provides alternative elementary proofs of the results on 3-cycles, complementing previous studies devoted to their analysis (see  \cite{saha}, \cite{bechhoefer}, \cite{gordon}, \cite{calvis}, \cite{burm}, \cite{zhang}, \cite{casu}).

At the end, the paper contains two appendices. The first is dedicated to the study of certain polynomials that appear in the theoretical results presented. The second appendix recalls some notions and results necessary for this work regarding the Schwarz derivative.

\section{Main results}

Consider the quadratic function $g:\R\to \R$ defined by $g(x)=ax^2+bx+c$, where $a,b,c\in \R$ and $a\neq 0$. The quadratic map generated by $g$ is defined by the recurrent relation:
\begin{equation}\label{r-r-123}
x_{t+1}=g(x_t),\,\,\,\forall t\in \mathbb{N}.
\end{equation}
We say that $x\in \R$ is a periodic point of period $p\in \mathbb{N}^*$ if $g^{(p)}(x)=x$ and $g^{(t)}(x)\neq x$ for all $1\leq t <p$. 

Let
$\mathcal{M}_p:=\{(x_1,\dots,x_p)\in \R^p\,|\,x_i\neq x_j\,\,\text{for all}\,\, 1\leq i<j\leq p\}$ be the set of ordered 
$p$-tuples of pairwise distinct real numbers.
On $\mathcal{M}_p$ we consider the equivalence relation $\sim_p$ defined by: $(x_1,\dots,x_p)\sim_p (y_1,\dots,y_p)$ if and only if $(y_1,\dots,y_p)$ is obtained from $(x_1,\dots,x_p)$ by a cyclic permutation. The corresponding quotient set is denoted by ${\bf M}_p=\mathcal{M}_p\,/\sim_p$.

A $p$-cycle of \eqref{r-r-123} is defined as an element of ${\bf M}_p$ that admits a representative $(x_1,\dots,x_p)\in \mathcal{M}_p$ satisfying  $g(x_1)=x_2, \dots, g(x_{p-1})=x_p, g(x_p)=x_1$.  
In particular, a 3-cycle of \eqref{r-r-123} is represented by the ordered triples $(x,y,z), (y,z,x), (z,x,y)\in \mathcal{M}_3$ which correspond to the same element of ${\bf M}_3$ and satisfy $g(x)=y, g(y)=z, g(z)=x$.

The points of a  3-cycle of \eqref{r-r-123} corresponding to a representative $(x,y,z)\in \mathcal{M}_3$ satisfy the system of equations:
\begin{equation}\label{3-cycle-231}
\begin{cases}
ax^2+bx+c=y \\
ay^2+by+c=z \\
az^2+bz+c=x.
\end{cases}
\end{equation}

It is easy to see that the following result holds.

\begin{prop}
Let $(x, y, z)\in \mathcal{M}_3$. Then there exists a unique quadratic map for which $(x, y, z)$ is a representative of a 3-cycle of \eqref{r-r-123}. The coefficients of this quadratic map are given by
$$
\begin{cases}
a=\frac{x^{2}+y^{2}+z^{2}-x y-x z-y z}{\left(y-z\right) \left(x-z\right) \left(x-y\right)}, \\
b=\frac{x^{3}+y^{3}+z^{3}-x^{2} z-x \,y^{2}-y \,z^{2}}{\left(y-z\right) \left(x-z\right) \left(x-y\right)},\\
c=\frac{x^{3} y+x \,z^{3}+y^{3} z-y^{2} z^{2}-x^{2} y^{2}-x^{2} z^{2}}{\left(y-z\right) \left(x-z\right) \left(x-y\right)}.
\end{cases}$$
\end{prop}

We introduce the bijective map $T:\mathcal{M}_3\to \R\times (\R\backslash \{0\})\times (\R\backslash \{0,-1\})$, defined by $T(x,y,z)=(x,p,r)$, where $p:=y-x$ and $r:=\frac{z-y}{y-x}$, so that  $r$ represents the {\bf ratio between successive distances}. 

If $(x, y, z)$ is a representative of a 3-cycle of \eqref{r-r-123} and $T(x,y,z)=(x,p,r)$, then 
\begin{equation}\label{abc-xpr-11}
\begin{cases}
a=-\frac{r^{2}+r+1}{p r \left(r+1\right)}, \\
b=\frac{pr^3 + 2pr^2 + 2r^2x + pr + 2rx + p + 2x}{pr(r + 1)}, \\
c=\frac{-p \,r^{3} x+p^{2} r^{2}-p \,r^{2} x-r^{2} x^{2}+p^{2} r-r \,x^{2}-p x-x^{2}}{p r \left(r+1\right)}.
\end{cases}
\end{equation}

In the sequel, we fix the quadratic function $g(x)=ax^2+bx+c$ and study the 3-cycles of \eqref{r-r-123}. {\bf The perturbed discriminant} 
\begin{equation}
\delta:=b^2-4ac-2b-7
\end{equation}
 plays a central role in this study. 

\begin{lem}\label{lemma-3452}
Let $(x,y,z)\in \mathcal{M}_3$ and set $(x,p,r)=T(x,y,z)$. The following statements are equivalent:
\begin{enumerate}[(i)]
\item $(x, y, z)$ is a representative of a 3-cycle of \eqref{r-r-123}.
\item The variables $p$, $x$, and $r$ satisfy
\begin{equation}\label{p-x-4701}
p= -\frac{r^{2}+r+1}{a r \left(r+1\right)},\,\,x=-\frac{b}{2a}+\frac{r^3+2r^2+r+1}{2ar(r+1)},
\end{equation}
and $r$ is a root of the polynomial $P_{\delta}$, defined in \eqref{P-9087}, where $\delta$ denotes the perturbed discriminant.
\end{enumerate}
\end{lem}

\begin{proof}
The implication $(i)\Rightarrow (ii)$ follows from \eqref{abc-xpr-11}. 

$(ii)\Rightarrow (i)$: Using $T^{-1}$ and the equality $P_{\delta}(r)=0$ we express $c$, $y$ and  $z$ as functions of $a$, $b$, and $r$. A direct computation then verifies that $(x,y,z)$ is a solution of \eqref{3-cycle-231}.
\end{proof}

To determine the points of a 3-cycle, we will use the family of polynomials $(Q_{\beta})_{\beta\in \R}$ defined and studied in Section \eqref{polynomials-0987}. The roots of a polynomial $Q_{\beta}$ are real and satisfy the inequalities $q^{\left<\beta\right>}_{1}<q^{\left<\beta\right>}_{2}<q^{\left<\beta\right>}_{3}$.

\begin{thm}\label{main-th-888}
The number of 3-cycles of the quadratic map \eqref{r-r-123} depends on the sign of the perturbed discriminant $\delta$ as follows:
\begin{enumerate}[(i)]
\item If $\delta<0$, the quadratic map has no 3-cycles.
\item If $\delta=0$, the quadratic map has a unique 3-cycle denoted by $\mathcal{C}_{\left<0\right>}$. 

The points of $\mathcal{C}_{\left<0\right>}$ are given by $x^{\left<0\right>}_{i}
=\frac{-b+1}{2a}+\frac{1}{a q^{\left<0\right>}_{i}}$,
 $i\in\{1,2,3\}.$
\item If $\delta>0$, the dynamical system has exactly two 3-cycles denoted by $\mathcal{C}_{\left<\delta\right>}$ and $\mathcal{C}_{\left<-\delta\right>}$. 

The points of $\mathcal{C}_{\left<\delta\right>}$ are $x^{\left<\delta\right>}_{i}=\frac{-b+\sqrt{\delta}+1}{2a}+\frac{1}{aq^{\left<\sqrt{\delta}\right>}_i}$, $i\in \{1,2,3\}$.

The points of $\mathcal{C}_{\left<-\delta\right>}$ are $x^{\left<-\delta\right>}_ i=\frac{-b-\sqrt{\delta}+1}{2a}+\frac{1}{aq^{\left<-\sqrt{\delta}\right>}_i}$, $i\in \{1,2,3\}$.
\end{enumerate}
\end{thm}

\begin{proof} $(i)$ We suppose that $(x,y,z)$ is a representative of a 3-cycle of \eqref{r-r-123} and let $(x,p,r)=T(x,y,z)$. By Lemma \ref{lemma-3452}, we obtain that $P_{\delta}(r)=0$ which contradicts Lemma \ref{P-prop-658}-$(i)$.

$(ii)$ From Lemma \ref{P-prop-658}-$(ii)$ we have $P_0=Q_0^2$. Let $r\in \R\backslash\{0,-1\}$ such that $Q_0(r)=0$ and let $p$ and $x$ be given by \eqref{p-x-4701}. By Lemma \ref{lemma-3452} the triple $(x,y,z)=T^{-1}(x,p,r)$ is a representative of a 3-cycle of \eqref{r-r-123}. 

Let $(y,\widetilde{p},\widetilde{r})=T(y,z,x)$ and $(z,\widehat{p},\widehat{r})=T(z,x,y)$. A direct computation yields $\widetilde {r}=\frac{x-z}{z-y}=\frac{x-y}{z-y}-1=-\frac{1}{r}-1$, $\widehat{r}=\frac{y-x}{x-z}=-\frac{1}{r+1}$. 
By Lemma \ref{Q-prop-995}-$(ii)$ the roots of $Q_0$ are precisely $r$, $\widetilde{r}$, and $\widehat{r}$.

Finally, using \eqref{p-x-4701} together with the relation $Q_0(r)=0$, we obtain the announced expressions for the points of the 3-cycle. 

$(iii)$ From Lemma \ref{P-prop-658}-$(ii)$ we have $P_{\delta}=Q_{\sqrt{\delta}}Q_{-\sqrt{\delta}}$. Let $r$ such that $Q_{\sqrt{\delta}}(r)=0$ and let $p$ and $x$ be given by \eqref{p-x-4701}. The triple $(x,y,z)=T^{-1}(x,p,r)$ is a representative of a 3-cycle of \eqref{r-r-123}. If $(y,\widetilde{p},\widetilde{r})=T(y,z,x)$ and $(z,\widehat{p},\widehat{r})=T(z,x,y)$, then the roots of $Q_{\sqrt{\delta}}$ are precisely $r$, $\widetilde{r}$, and $\widehat{r}$. The triples $(x,y,z)$, $(y,z,x)$, and $(z,y,x)$ are representatives of the same 3-cycle, denoted by $\mathcal{C}_{\left<\delta\right>}$. Using \eqref{p-x-4701} and the relation $Q_{\sqrt{\delta}}(r)=0$, we obtain the announced expressions for the points of the 3-cycle. 

Similarly, using the roots of $Q_{-\sqrt{\delta}}$, we obtain the 3-cycle $\mathcal{C}_{\left<-\delta\right>}$ and the expressions for its points.
\end{proof}

\begin{rem}
Let $u<v<w$ be the points of a 3-cycle $\mathcal{C}$ of \eqref{r-r-123}. Then $(u, w, v)$ is a representative of $\mathcal{C}$ if  $a>0$, while  $(u, v, w)$ is a representative if $a<0$.

In our setting, for $\delta\geq 0$, using the ordering of the roots of the polynomial $Q_{\beta}$, we obtain:
\begin{enumerate}[(i)]
\item If $a>0$, then $x^{\left<\pm \delta\right>}_2<x^{\left<\pm \delta\right>}_1<x^{\left<\pm \delta\right>}_3$.
\item If $a<0$, then $x^{\left<\pm \delta\right>}_3<x^{\left<\pm \delta\right>}_1<x^{\left<\pm \delta\right>}_2$.
\end{enumerate}
In both cases the ordered triple $(x^{\left<\pm \delta\right>}_1, x^{\left<\pm \delta\right>}_2, x^{\left<\pm \delta\right>}_3)$ is a representative of the 3-cycle $\mathcal{C}_{\left<\pm\delta\right>}$.
\end{rem}

\begin{rem}\label{polynomials-points-953}
Taking into account the above expressions of the points in $\mathcal{C}_{\left<\delta\right>}$, respectively $\mathcal{C}_{\left<-\delta\right>}$, and the fact that $q^{\left<\pm\sqrt{\delta}\right>}_i$ are the roots of the polynomials $Q_{\pm\sqrt{\delta}}$, we deduce that the points in $\mathcal{C}_{\left<\delta\right>}$, respectively $\mathcal{C}_{\left<-\delta\right>}$, are the roots of the following cubic polynomials:
$$R_{\pm\delta}(X)=X^3+c^{\left<\pm\delta\right>}_2 X^2+c^{\left<\pm\delta\right>}_1 X+c^{\left<\pm\delta\right>}_0,$$
where
$$\begin{cases}
c^{\left<\pm\delta\right>}_2 = \frac{1}{2a}(3b+1\mp\sqrt{\delta}), \\
c^{\left<\pm\delta\right>}_1=\frac{1}{4a^2}(3b^2+2b-9-\delta\mp 2(b+1)\sqrt{\delta}), \\
c^{\left<\pm\delta\right>}_0=\frac{1}{8a^3}(\pm \sqrt{\delta^3}+(1-b)\delta\mp(b^2+2b-7)\sqrt{\delta}+b^3+b^2-9b-1).
\end{cases}
$$
\end{rem}

The ratio between successive distances characterizes whether an ordered $3$-cycle is in $\mathcal{C}_{\left<\delta\right>}$ or $\mathcal{C}_{\left<-\delta\right>}$.

\begin{thm}
Let $(x,y,z)$ be a representative of a $3$-cycle of \eqref{r-r-123} and
$r=\frac{z-y}{y-x}$  the ratio between successive distances.
\begin{enumerate}[(i)]
\item $\delta=0 \Leftrightarrow 
r\in\{q^{\left<0\right>}_{1}, q^{\left<0\right>}_{2}, q^{\left<0\right>}_{3}\}$, where $q^{\left<0\right>}_{1}\approx -1.801$, $q^{\left<0\right>}_{2}\approx -0.445$, $q^{\left<0\right>}_{3}\approx 1.246$.

\item $\delta>0$ and $(x,y,z)\in\mathcal{C}_{\left<\delta\right>}$ if and only if
$
r\in(q^{\left<0\right>}_{1},-1)\cup(q^{\left<0\right>}_{2},0)\cup(q^{\left<0\right>}_{3},\infty).
$

\item $\delta>0$ and $(x,y,z)\in\mathcal{C}_{\left<-\delta\right>}$ if and only if
$
r\in(-\infty,q^{\left<0\right>}_{1})\cup(-1,q^{\left<0\right>}_{2})\cup(0,q^{\left<0\right>}_{3}).
$
\end{enumerate}
\end{thm}

\begin{proof}
The existence of a 3-cycle for \eqref{r-r-123} implies that $\delta\ge 0$. By Theorem \ref{main-th-888}, there exists $\beta\in \{\sqrt{\delta}, -\sqrt{\delta}\}$ such that $Q_{\beta}(r)=0$. Using the notation from the proof of Lemma \ref{Q-prop-995}, we obtain that $\beta=h(r)=\frac{Q_0(r)}{r(r+1)}$. Since $h$ is strictly increasing on each interval $(-\infty,-1)$, $(-1,0)$ and $(0,\infty)$, the result follows. 
\end{proof}

\medskip

In what follows, we assume that $\delta\ge 0$ and study the stability of 3-cycles of \eqref{r-r-123}. 
Let $\mathcal{C}$ be a 3-cycle of \eqref{r-r-123} and let $x,y,z$ denote its points. We recall some stability concepts that will be used below. For more details, see \cite{elaydi}.

\noindent $\bullet$ {\it Hyperbolicity}: The cycle $\mathcal{C}$ is said to be hyperbolic if $H(\mathcal{C}):=|g'(x)g'(y)g'(z)|\neq 1$. 

\noindent $\bullet$ {\it Stability}: The cycle $\mathcal{C}$ is stable if one (and, equivalently, any) of its points is a Lyapunov stable fixed point of $g^{(3)}$. More precisely, $\mathcal{C}$ is stable if, for any $\varepsilon>0$, there exists $\delta>0$ such that for all $x_0\in \R$ with $|x_0-x|<\delta$ we have $|g^{(3t)}(x_0)-x|<\varepsilon$ for all $t\in \mathbb{N}$. 

\noindent $\bullet$ {\it Asymptotic stability}: The cycle $\mathcal{C}$ is asymptotically stable if one (and, equivalently, any) of its points is asymptotically stable fixed point of $g^{(3)}$. $\mathcal{C}$ is asymptotically stable if it is stable and there is $\gamma>0$ such that for all $x_0\in \R$ with $|x_0-x|<\gamma$ we have $\lim_{t\to \infty} g^{(3t)}(x_0)=x$.

\begin{lem}\label{hyperbolicity-997}
Let $\delta\ge 0$. Then:
\begin{enumerate}[(i)]
\item If $\delta=0$, the 3-cycle $\mathcal{C}_{\left<0\right>}$ is nonhyperbolic.
\item If $\delta>0$, the 3-cycle $\mathcal{C}_{\langle -\delta \rangle}$ is hyperbolic, while $\mathcal{C}_{\left<\delta\right>}$ is nonhyperbolic if and only if 
$$\delta=\delta_{nh}:=\frac{\left(\left(460+60 \sqrt{201}\right)^{\frac{2}{3}}-2 \left(460+60 \sqrt{201}\right)^{\frac{1}{3}}-80\right)^{2}}{36 \left(460+60 \sqrt{201}\right)^{\frac{2}{3}}}\approx 0.074.$$
\end{enumerate}
\end{lem}

\begin{proof}
From Lemma \ref{propr-870} we deduce that $H(\mathcal{C}_{\left<\delta\right>})=-\delta^{\frac{3}{2}}-\delta-7\delta^{\frac{1}{2}}+1$ and $H(\mathcal{C}_{\left<-\delta\right>})=\delta^{\frac{3}{2}}-\delta+7\delta^{\frac{1}{2}}+1$. A straightforward analysis proves the stated results.
\end{proof}

\begin{thm}\label{main-stability-777}
For $\delta\ge 0$, the stability of the 3-cycles of \eqref{r-r-123} are characterized as follows. 
\begin{enumerate}[(i)]
\item If $\delta\in (\delta_{nh},\infty)$, then both $\mathcal{C}_{\left<\delta\right>}$ and $\mathcal{C}_{\left<-\delta\right>}$ are unstable.
\item If $\delta\in (0,\delta_{nh}]$, then $\mathcal{C}_{\left<\delta\right>}$ is asymptotically stable, while $\mathcal{C}_{\left<-\delta\right>}$ is unstable.
\item If $\delta=0$, then $\mathcal{C}_{\left<0\right>}$ is unstable. 
\end{enumerate}
\end{thm}

\begin{proof}
$(i)$ As shown in Lemma \ref{hyperbolicity-997}, we have $|H(\mathcal{C}_{\left<\delta\right>})|=|-\delta^{\frac{3}{2}}-\delta-7\delta^{\frac{1}{2}}+1|>1$ and $|H(\mathcal{C}_{\left<-\delta\right>})|=|\delta^{\frac{3}{2}}-\delta+7\delta^{\frac{1}{2}}+1|>1$.  Hence, by Theorem 1.7, \cite{elaydi}, both $\mathcal{C}_{\left<\delta\right>}$ and $\mathcal{C}_{\left<-\delta\right>}$ are unstable.

$(ii)$ As in the previous case, here too we have $|H(\mathcal{C}_{\left<-\delta\right>})|>1$ which demonstrates the instability of $\mathcal{C}_{\left<-\delta\right>}$. 

If $\delta\in (0,\delta_{nh})$, then $|H(\mathcal{C}_{\left<\delta\right>})|<1$, and by Theorem 1.7, \cite{elaydi}, $\mathcal{C}_{\left<\delta\right>}$ is asymptotically stable. 
In the case $\delta=\delta_{nh}$ the 3-cycle $\mathcal{C}_{\left<\delta\right>}$ is nonhyperbolic and its study is carried out using the Schwarzian derivative (see Appendix \ref{Schwarzian-derivative}). A direct computation shows that $\mathcal{S}g(x)=-\frac{6a^2}{(2ax+b)^2}<0$, and hence $\mathcal{S}g<0$. From Lemma \ref{schwarzin-composition-23} we deduce that 
$\mathcal{S}(g^{(3)})<0$. If $y$ is a point of $\mathcal{C}_{\left<\delta_{nh}\right>}$, then it is a fixed point of $g^{(3)}$. By Theorem 1.6 in \cite{elaydi}, we can conclude that $y$ is asymptotically stable for the discrete dynamical system generated by $g^{(3)}$. Consequently, $\mathcal{C}_{\left<\delta_{nh}\right>}$ is asymptotically stable.

$(iii)$ Lemma \ref{hyperbolicity-997} implies that $\mathcal{C}_{\left<0\right>}$ is nonhyperbolic. In this case, we have $c=\frac{b^2-2b-7}{4a}$ and
$g^{(3)}(x)-x=\frac{1}{256 a}P_1(x)P_2^2(x),$
where
\begin{align*}
P_1(x) & =4a^2x^2+4a(b-1)x+b^2-2b-7, \\
P_2(x) & =8a^3x^3+4a^2(3b+1)x^2+2a(3b^2+2b-9)x+b^3+b^2-9b-1.
\end{align*}
The roots of $P_1$ are equilibria of \eqref{r-r-123}. Let $y$ be a point of $\mathcal{C}_{\left<0\right>}$. Then $P_2(y)=0$, $P_1(y)\neq 0$, $(g^{(3)})'(y)=1$, and $(g^{(3)})''(y)=\frac{1}{128 a}P_1(y)(P_{2}'(y))^2$. Since $P_2$ have three distinct roots, we have that $P_2'(y)\neq 0$. By Theorem 1.5 in \cite{elaydi} it follows that $y$ is an unstable fixed point of the discrete dynamical system generated by $g^{(3)}$. Consequently, $\mathcal{C}_{\left<0\right>}$ is unstable.  
\end{proof}

The stability of $\mathcal{C}_{\left<\delta\right>}$ can be determined by examining the ratio between successive distances.

\begin{thm}
Assume that $\delta>0$. Let $(x,y,z)$ be a representative of $\mathcal{C}_{\left<\delta\right>}$, and let
$r=\frac{z-y}{y-x}$ be the ratio between successive distances. The following statements are equivalent. 
\begin{enumerate}[(i)]
\item $\mathcal{C}_{\left<\delta\right>}$ is asymptotically stable;
\item $r\in (-1.801..,-1.713..]\cup(-0.445.., -0.416..]\cup (1.246.., 1.401..]$.
\end{enumerate}
\end{thm}

\begin{proof} By Theorem \ref{main-stability-777} the 3-cycle $\mathcal{C}_{\left<\delta\right>}$ is asymptotically stable if and only if $\delta\in (0,\delta_{nh}]$. 
Using the notation from the proof of Lemma \ref{Q-prop-995}, we have $\sqrt{\delta}=h(r)=\frac{Q_0(r)}{r(r+1)}$. Since $h$ is strictly increasing on each interval $(-\infty,-1)$, $(-1,0)$ and $(0,\infty)$, it follows that $\delta\in (0,\delta_{nh}]$ if and only if $r\in (q^{\left<0\right>}_1, r_{nh,1}]\cup (q^{\left<0\right>}_2, r_{nh,2}]\cup (q^{\left<0\right>}_3, r_{nh,3}]$, where $q^{\left<0\right>}_1, q^{\left<0\right>}_2, q^{\left<0\right>}_3$ are the roots of $Q_0$, and $r_{nh,1}=-1.713..$, $r_{nh,2}=-0.416..$, and $r_{nh,3}=1.401..$ are the solutions of the equation $h(x)=\sqrt{\delta_{nh}}$.
\end{proof}

\section{Selected classical cases}

{(\bf I)} The discrete dynamical system 
\begin{equation}\label{casu-910}
z_{t+1}=z_t^2+c, \,\,\,\forall t\in \mathbb{N}
\end{equation}
originates in the early 20th-century work of Gaston Julia and Pierre Fatou on the iteration of rational functions, which led to the concepts of Julia and Fatou sets (see \cite{julia}, \cite{fatou}, and \cite{milnor}).
Benoît Mandelbrot later revealed the global structure of this family through computer experimentation, introducing the Mandelbrot set and establishing the system as a central object in complex dynamics (see \cite{mandelbrot} and \cite{hubbard}).
In \cite{casu}, the 3-cycles of the discrete dynamical system \eqref{casu-910} are analysed over the set of real numbers.

In this case, the perturbed discriminant is $\delta=-4c-7$. From Theorem \ref{main-th-888}, we obtain the following classification:
\begin{itemize}
\item If $c>-\frac{7}{4}$, then \eqref{casu-910} has no 3-cycles.
\item If $c=-\frac{7}{4}$, then \eqref{casu-910} has a unique 3-cycle, denoted by $\mathcal{C}_{\left<0\right>}$, which consists of the following points: $x^{\left<0\right>}_1=\frac{1}{2}+\frac{1}{q^{\left<0\right>}_1}\approx -0.054$, $x^{\left<0\right>}_2=\frac{1}{2}+\frac{1}{q^{\left<0\right>}_2}\approx -1.746$, and $x^{\left<0\right>}_3=\frac{1}{2}+\frac{1}{q^{\left<0\right>}_3}\approx 1.301$. 
\item If $c<-\frac{7}{4}$, then \eqref{casu-910} has two 3-cycles, denoted by $\mathcal{C}_{\left<-4c-7\right>}$ and $\mathcal{C}_{\left<4c+7\right>}$. 

The points of $\mathcal{C}_{\left<-4c-7\right>}$ are given by $x^{\left<-4c-7\right>}_i=\frac{\sqrt{-4c-7}+1}{2}+\frac{1}{q^{\left<\sqrt{-4c-7}\right>}_i}$. 

The points of $\mathcal{C}_{\left<4c+7\right>}$ are given by $x^{\left<4c+7\right>}_i=\frac{-\sqrt{-4c-7}+1}{2}+\frac{1}{q^{\left<-\sqrt{-4c-7}\right>}_i}$. 
\end{itemize}

By Remark \ref{polynomials-points-953} we can write the polynomials with the roots the points of 3-cycles.
\begin{itemize}
\item $R_0(X)=X^{3}+\frac{1}{2} X^{2}-\frac{9}{4} X -\frac{1}{8}$ is the polynomial which has the roots exactly the points of $\mathcal{C}_0$.
\item $R_{-4c-7}(X)=X^{3}+\left(\frac{1}{2}-\frac{\sqrt{-4 c -7}}{2}\right) X^{2}+\left(-\frac{1}{2}+c +\frac{\sqrt{-4 c -7}}{2}\right) X -\frac{\sqrt{-4 c -7}\, c}{2}-\frac{c}{2}-1$ is the polynomial which has the roots exactly the points of $\mathcal{C}_{\left<-4c-7\right>}$.
\item $R_{4c+7}(X)=X^{3}+\left(\frac{1}{2}+\frac{\sqrt{-4 c -7}}{2}\right)X^{2} +\left(-\frac{1}{2}+c +\frac{\sqrt{-4 c -7}}{2}\right) X +\frac{\sqrt{-4 c -7}\, c}{2}-\frac{c}{2}-1$ is the polynomial which has the roots exactly the points of $\mathcal{C}_{\left<4c+7\right>}$.
\end{itemize}

By Lemma \ref{hyperbolicity-997}, the cycle $\mathcal{C}_{\delta}$ is nonhyperbolic if and only if $c=-\frac{7}{4}$ or $c=c_{nh}\approx -1.768$.
When $c\leq -\frac{7}{4}$, Lemma \ref{main-stability-777} implies that the stability of the 3-cycles of \eqref{casu-910} is as follows:
\begin{itemize}
\item If $c\in (-\infty, c_{nh})$, then both  $\mathcal{C}_{\left<-4c-7\right>}$ and $\mathcal{C}_{\left<4c+7\right>}$ are unstable.
\item If $c\in [c_{nh}, -\frac{7}{4})$, then $\mathcal{C}_{\left<-4c-7\right>}$ is asymptotically stable and $\mathcal{C}_{\left<4c+7\right>}$ is unstable.
\item If $c=-\frac{7}{4}$, then $\mathcal{C}_{\left<0\right>}$ is unstable. 
\end{itemize}
Above, we have recovered the results presented in \cite{casu} and have found new formulas for the points of the 3-cycles of \eqref{casu-910} using the roots of the polynomials $Q_{\sqrt{-4c-7}}$ and $Q_{-\sqrt{-4c-7}}$.
 \medskip

{(\bf II)} The logistic map is perhaps the most widely recognized example of a simple nonlinear system. Its origins trace back to P. F. Verhulst, who formulated an elementary model of population growth (see \cite{verhulst}). Over a century later, R. May revisited the model in a discrete form, showing that even such a simple iterative process can produce highly intricate and unpredictable behaviour (see \cite{may}). May’s work was instrumental in establishing the logistic map as a central example in the study of dynamical systems.

The logistic map on $\R$ is defined by the recurrence relation
\begin{equation}\label{logistic-9945}
x_{t+1}=\lambda x_t (1-x_t), \,\,\,t\in \mathbb{N},
\end{equation}
where $\lambda$ is a real parameter. 
The perturbed discriminant is $\delta=\lambda^2-2\lambda-7$ and, by using Theorem \ref{main-th-888}, we obtain the following classification:
\begin{itemize}
\item If $\lambda\in (1-2\sqrt{2},1+2\sqrt{2})$, then \eqref{logistic-9945} has no 3-cycles.
\item If $\lambda\in \{1-2\sqrt{2},1+2\sqrt{2}\}$, then \eqref{logistic-9945} has a unique 3-cycle, denoted by $\mathcal{C}_{\left<0\right>}$. The points of this cycle are:
\begin{itemize}
\item for $\lambda=1+2\sqrt{2}$,  $x^{\left<0\right>}_1\approx 0.514, x^{\left<0\right>}_2\approx 0.956, x^{\left<0\right>}_3\approx 0.159$.
\item for $\lambda=1-2\sqrt{2}$,  $x^{\left<0\right>}_1\approx 0.469, x^{\left<0\right>}_2\approx -0.455, x^{\left<0\right>}_3\approx 1.212$.
\end{itemize}

\item If $\lambda\in (-\infty, 1-2\sqrt{2})\cup (1+2\sqrt{2}, \infty)$, then \eqref{logistic-9945} has two 3-cycles, denoted by $\mathcal{C}_{\left<\lambda^2-2\lambda-7\right>}$ and $\mathcal{C}_{\left<-\lambda^2+2\lambda+7\right>}$. 

The points of $\mathcal{C}_{\left<\lambda^2-2\lambda-7\right>}$ are given by $x^{\left<\lambda^2-2\lambda-7\right>}_i=\frac{\lambda-\sqrt{\lambda^2-2\lambda-7}-1}{2\lambda}-\frac{1}{\lambda q^{\left<\sqrt{\lambda^2-2\lambda-7}\right>}_i}$. 

The points of $\mathcal{C}_{\left<-\lambda^2+2\lambda+7\right>}$ are given by $x^{\left<-\lambda^2+2\lambda+7\right>}_i=\frac{\lambda+\sqrt{\lambda^2-2\lambda-7}-1}{2\lambda}-\frac{1}{\lambda q^{\left<-\sqrt{\lambda^2-2\lambda-7}\right>}_i}$.  
\end{itemize}

By Remark \ref{polynomials-points-953} we can write the polynomials with the roots the points of 3-cycles. Due to the large size of the expressions of these polynomials, we omit writing them.

By Lemma \ref{hyperbolicity-997}, the cycle $\mathcal{C}_{\left<\delta\right>}$ is nonhyperbolic if and only if $\lambda\in \{1-2\sqrt{2},1+2\sqrt{2}\}$ or $\lambda^2-2\lambda-7=\delta_{nh}\Leftrightarrow \lambda\in \{-1.841.., 3.841..\}$.

When $\lambda\in (-\infty, 1-2\sqrt{2}]\cup [1+2\sqrt{2}, \infty)$, Lemma \ref{main-stability-777} implies that the stability of the 3-cycles of \eqref{logistic-9945} is as follows:
\begin{itemize}
\item If $\lambda\in (-\infty, -1.841..)\cup (3.841..,\infty)$, then both cycles $\mathcal{C}_{\left<\lambda^2-2\lambda-7\right>}$ and $\mathcal{C}_{\left<-\lambda^2+2\lambda+7\right>}$ are unstable.
\item If $\lambda\in [-1.841..,1-2\sqrt{2})\cup (1+2\sqrt{2}, 3.841..]$ then $\mathcal{C}_{\left<\lambda^2-2\lambda-7\right>}$ is asymptotically stable and $\mathcal{C}_{\left<-\lambda^2+2\lambda+7\right>}$ is unstable.
\item If $\lambda\in \{1-2\sqrt{2},1+2\sqrt{2}\}$, then $\mathcal{C}_{\left<0\right>}$ is unstable.
\end{itemize}

In the study of the logistic map for $\lambda>0$, the papers \cite{saha}, \cite{bechhoefer}, \cite{gordon}, \cite{calvis}, \cite{burm} and \cite{zhang} provide elementary proofs that $1+2\sqrt{2}$ represents the smallest value of the parameter $\lambda$ for which 3-cycles occur. In \cite{gordon} it is also shown that $3.841..$ is the largest value of $\lambda$ for which a stable 3-cycle exists.

\appendix

\section{Key polynomials}\label{polynomials-0987}
We consider the following families of polynomials:
\begin{equation}\label{Q-def-321}
Q_{\beta}(X)=X^3+(1-\beta)X^2-(2+\beta)X-1,\,\,\beta\in \R
\end{equation}
and
\begin{equation}\label{P-9087}
P_{\alpha}(X):=X^6+2X^5-(\alpha+3) X^4-2(\alpha+3) X^3+(2-\alpha)X^2+4X+1,\,\,\alpha\in \R.
\end{equation}

We present the properties of interest of these families of polynomials.

\begin{lem}\label{Q-prop-995}
Let be $\beta, \gamma\in \R$. 
\begin{enumerate}[(i)]
\item All roots of $Q_{\beta}$ are real, and their ordering satisfies $q^{\left<\beta\right>}_1<-1<q^{\left<\beta\right>}_2<0<q^{\left<\beta\right>}_3.$ 
\item If $q$ is a root of $Q_{\beta}$, then other two roots are $-\frac{1}{q}-1$ and $-\frac{1}{q+1}$.
\item If $\beta<\gamma$, then $q^{\left<\beta\right>}_1<q^{\left<\gamma\right>}_1<-1<q^{\left<\beta\right>}_2<q^{\left<\gamma\right>}_2<0<q^{\left<\beta\right>}_3<q^{\left<\gamma\right>}_3.$
\end{enumerate}
\end{lem}

\begin{proof}
$(i)$ The following equalities: $Q_{\beta}(-1)=1$, $Q_{\beta}(0)=-1$, $\lim_{x\to \infty}Q_{\beta}(x)=\infty$ and  $\lim_{x\to -\infty}Q_{\beta}(x)=-\infty$ establish the claimed result. 

$(iii)$ Since $Q_{\beta}(X)=Q_0(X)-\beta X(X+1)$, we define $h:\R\backslash \{0,-1\}\to \R$, $h(x)=\frac{Q_0(x)}{x(x+1)}$. Its derivative $h'(x)=\frac{(x^2+x+1)^2}{x^2(x+1)^2}>0$ shows that $h$ is strictly increasing on each interval $(-\infty,-1)$, $(-1,0)$ and $(0,\infty)$.
\end{proof}

\begin{rem}\label{Q0-124}
The roots of $Q_0(X)=X^3+X^2-2X-1$ are $q^{\left<0\right>}_1\approx -1.801$, $q^{\left<0\right>}_2\approx -0.445$, and $q^{\left<0\right>}_3\approx 1.246$.
\end{rem}

Using the identity $P_{\alpha}(X)=(X^3+X^2-2X-1)^2-\alpha X^2(X+1)^2$ together with the above Lemma, we obtain the following results concerning the roots of $P_{\alpha}$.
\begin{lem}\label{P-prop-658}
Let $\alpha\in \R$. We distinguish the following cases.
\begin{enumerate}[(i)]
\item If $\alpha<0$, then $P_{\alpha}$ have no real roots.
\item If $\alpha\geq 0$, then $P_{\alpha}=Q_{\sqrt{\alpha}}Q_{-\sqrt{\alpha}}$ and all its roots are real. 
\end{enumerate}
\end{lem}

\begin{lem}\label{propr-870}
Let $\beta\in \R$, and let $q^{\left<\beta\right>}_1, q^{\left<\beta\right>}_2, q^{\left<\beta\right>}_3$ be the roots of $Q_{\beta}$. Define $x^{\left<\beta\right>}_i=\frac{-b+\beta+1}{2a}+\frac{1}{aq^{\left<\beta\right>}_i}$, $i\in \{1,2,3\}$. Then 
$$H:=(2ax^{\left<\beta\right>}_1+b)(2ax^{\left<\beta\right>}_2+b)(2ax^{\left<\beta\right>}_3+b)=-\beta^3-\beta^2-7\beta+1.$$
\end{lem}

\begin{proof}
By direct computations we have:
\begin{align*}
H=\beta^3+(3+\frac{2s_2}{s_3})\beta^2+(3+\frac{4s_2}{s_3}+\frac{4s_1}{s_3})\beta+1+\frac{2s_2}{s_3}+\frac{4s_1}{s_3}+\frac{8}{s_3},
\end{align*}
where
$$\begin{cases}
s_1=q^{\left<\beta\right>}_1+q^{\left<\beta\right>}_2+ q^{\left<\beta\right>}_3 \\
s_2=q^{\left<\beta\right>}_1 q^{\left<\beta\right>}_2+ q^{\left<\beta\right>}_1 q^{\left<\beta\right>}_3+q^{\left<\beta\right>}_2 q^{\left<\beta\right>}_3 \\
s_3=q^{\left<\beta\right>}_1 q^{\left<\beta\right>}_2 q^{\left<\beta\right>}_3.
\end{cases}$$
 By Vieta's relations we have, $s_1=\beta-1$, $s_2=-\beta-2$, and $s_3=1$ which immediately implies the claimed result.
\end{proof}

\section{Schwarzian derivative}\label{Schwarzian-derivative}

The following notions and results are presented following the book \cite{devaney}.

Let $f:\R\to \R$ a nonconstant $\mathcal{C}^3$ map. The Schwarzian derivative of $f$ is the map $\mathcal{S}f:\R\backslash \{x\in \R\,|\,f'(x)=0\}\to \R$ defined by
\begin{equation}
\mathcal{S}f(x)=\frac{f'''(x)}{f'(x)}-\frac{3}{2}\left(\frac{f''(x)}{f'(x)}\right)^2.
\end{equation}
The chain rule for Schwarzian derivatives states that:
$$\mathcal{S}(f\circ h)(x)=\mathcal{S}f(h(x))+\mathcal{S}h(x),$$
where $f,h:\R\to \R$ are $\mathcal{C}^3$ maps and the expressions are defined at points where the derivatives do not vanish.

We say that $\mathcal{S}f<0$ if $\mathcal{S}f(x)<0$ for all $x$ in the domain of $\mathcal{S}f$.
\begin{lem}\label{schwarzin-composition-23}
If $\mathcal{S}f<0$ and $n\in \mathbb{N}^*$, then $\mathcal{S}(f^{(n)})<0$. 
\end{lem}

\end{document}